\documentclass[letterpaper,10pt,pdflatex,reqno]{amsart}
\usepackage{bbm}
\usepackage{amsmath,times}
\usepackage{amsthm}
\usepackage{amssymb}
\usepackage{latexsym}
\usepackage{amscd}
\usepackage[english]{babel}
\usepackage[latin1]{inputenc}
\usepackage{graphicx}
\usepackage{color}
\usepackage{verbatim}
\newtheorem{theorem}{Theorem}[section]
\newtheorem{lemma}[theorem]{Lemma}
\newtheorem{proposition}[theorem]{Proposition}

\theoremstyle{definition}
\newtheorem{definition}[theorem]{Definition}
\newtheorem{corollary}[theorem]{Corollary}
\newtheorem{remark}[theorem]{Remark}

\begin{document} 

\title{Monotonicity of non-pluripolar Monge-Amp\`ere masses}
\author{David Witt Nystr\"om}
\maketitle

\begin{abstract}
We prove that on a compact K\"ahler manifold, the non-pluripolar Monge-Amp\`ere mass of a $\theta$-psh function decreases as the singularities increase. This was conjectured by Boucksom-Eyssidieux-Guedj-Zeriahi who proved it under the additional assumption of the functions having small unbounded locus. As a corollary we get a comparison principle for $\theta$-psh functions, analogous to the comparison principle for psh functions due to Bedford-Taylor.
\end{abstract}

\section{Introduction}

Let $(X,\omega)$ be a compact K\"ahler manifold and $\theta$ be a smooth real $(1,1)$-form on $X$. Recall that a function $\phi:X\to [-\infty,\infty)$ is said to be $\theta$-psh if it is upper semicontinuous, $L^1_{loc}$ and $\theta+dd^c\phi\geq 0$ in the sense of currents. Note that by the $dd^c$-lemma any closed positive $(1,1)$-current $T$ in the class $[\theta]$ can be written as $T=\theta+dd^c\phi$ where $\phi$ is $\theta$-psh.

When $\phi$ is smooth one defines the Monge-Amp\`ere measure $MA_{\theta}(\phi):=(\theta+dd^c\phi)^n$, which will be a semipositive $(n,n)$-form, and clearly $$\int_X MA_{\theta}(\phi)=\int_X\theta^n.$$ By the fundamental work of Bedford-Taylor \cite{BT82,BT87} one can still define a Monge-Amp\`ere measure $MA_{\theta}(\phi)$, which will be a finite positive measure, only assuming $\phi$ to be locally bounded. Even when $\phi$ is not locally bounded one can define a positive measure $MA_{\theta}(\phi)$ called the non-pluripolar Monge-Amp\`ere measure, but a priori this could fail to be locally finite. However, Boucksom-Eyssidieux-Guedj-Zeriahi established in \cite{BEGZ} the finiteness of $MA_{\theta}(\phi)$, the mass being bounded by the volume of the class $[\theta]$ \cite{BEGZ} (for $\theta$ K\"ahler this was first shown by Guedj-Zeriahi \cite{GZ}). But in contrast to the case when $\phi$ is smooth (or locally bounded) the mass $\int_X MA_{\theta}(\phi)$ does not only depend on $[\theta]$. To describe the dependence of the Monge-Amp\`ere mass on $\phi$ we recall the following definition.

\begin{definition}
If $\phi$ and $\psi$ are two $\theta$-psh functions we say that $\phi$ is less singular than $\psi$ if $\phi\geq \psi+O(1)$.  
\end{definition}

Also recall that $\phi$ is said to have small unbounded locus if it is locally bounded away from a closed complete pluripolar subset of $X$. In \cite[Thm 1.16]{BEGZ} Boucksom-Eyssidieux-Guedj-Zeriahi proved that if $\phi$ is less singular than $\psi$ and both $\phi$ and $\psi$ have small unbounded locus then $$\int_X MA_{\theta}(\phi)\geq \int_X MA_{\theta}(\psi).$$ They conjectured this to be true even without the assumption of small unbounded locus. In this paper we prove their conjecture, i.e.:

\begin{theorem} \label{compthm}
Let $\phi$ and $\psi$ be two $\theta$-psh functions. If $\phi$ is less singular than $\psi$ then 
\begin{equation} \label{introineq}
\int_X MA_{\theta}(\phi)\geq \int_X MA_{\theta}(\psi).
\end{equation} 
\end{theorem}

\begin{remark}
In fact, \cite[Thm 1.16]{BEGZ} contains a more general statement about the pseudoeffectivity of certain differences of cohomology classes of non-pluripolar products (see Theorem \ref{BEGZthm}). 
\end{remark}

Let us the recall the fundamental comparison principle for $psh$ functions due to Bedford-Taylor \cite[Thm. 4.1]{BT82}:

\begin{theorem} \label{BTcomp}
Let $\Omega$ be a bounded open set in $\mathbb{C}^n$ and $u$ and $v$ two locally bounded $psh$ functions on $U$ such that $\liminf_{z\to \partial \Omega}u(z)-v(z)\geq 0$ (i.e. $u\geq v$ on $\partial \Omega$). Then $$\int_{\{u<v\}}MA(v)\leq \int_{\{u<v\}}MA(u).$$ 
\end{theorem}

As was observed in \cite{BEGZ}, Theorem \ref{compthm} implies an analogous comparison principle for $\theta$-psh functions:

\begin{corollary} \label{corcomp}
Let $\phi$ and $\psi$ be $\theta$-psh and assume $\phi$ is less singular than $\psi$. Then $$\int_{\{\phi<\psi\}}MA_{\theta}(\psi)\leq \int_{\{\phi<\psi\}}MA_{\theta}(\phi).$$
\end{corollary}

In order to prove Theorem \ref{compthm} we will, given $\theta$, $\phi$ and $\psi$, construct a related form $\tilde{\theta}$ together with $\tilde{\theta}$-psh functions $\Phi$ and $\Psi$ on $X\times \mathbb{P}^N$ for large $N$. By construction $\Phi$ will be less singular than $\Psi$ and they will also have small unbounded locus so we know from \cite[Thm 1.16]{BEGZ} that 
\begin{equation} \label{introeq}
\int_{X\times \mathbb{P}^N}MA_{\tilde{\theta}}(\Phi)\geq \int_{X\times \mathbb{P}^N}MA_{\tilde{\theta}}(\Psi).
\end{equation} 
We will then establish a formula for the Monge-Amp\`ere masses of $\Phi$ and $\Psi$ involving the Monge-Amp\`ere masses of $\phi$ and $\psi$ (Prop. \ref{keyprop}), so that invoking (\ref{introeq}) for larger and larger $N$ yields the desired inequality (\ref{introineq}).    

\subsection{Related work}

In \cite{Darvas} Darvas proved that if $\theta$ is K\"ahler then a $\theta$-psh function $\psi$ has full Monge-Amp\`ere mass (i.e. $\int_X MA_{\theta}(\psi)=\int_X\theta^n$) iff whenever $\phi$ is $\theta$-psh and locally bounded we have that $P_{[\psi]}(\phi)=\phi$. Here $P_{[\psi]}(\phi)$ is a certain kind of envelope introduced in \cite{RWN}, defined as the usc regularization of the supremum of all $\theta$-psh functions $\phi'$ such that $\phi'\leq \phi$ and $\phi'\leq \psi+O(1)$. Recently Darvas-Di Nezza-Lu generalized this result to the case when $[\theta]$ is just big. 

These results have had important applications to the study of the geometry of the space of full mass currents, but are not in an obvious way connected to the monotonicity on the Monge-Amp\`ere masses. However, interestingly their proof uses so called geodesic rays constructed from the $\theta$-psh function $\psi$. A geodesic ray can be thought of as a $\pi_X^*\theta$-psh function on $X\times \mathbb{D}$, where $\mathbb{D}$ is the unit disc, and one could also extend it to a $\tilde{\theta}:=\pi_X^*\theta+\pi_{\mathbb{P}^1}\omega_{FS}$-psh function on $X\times \mathbb{P}^1$. This is similar to our construction of a $\tilde{\theta}$-psh functions $\Phi$ on $X\times \mathbb{P}^N$ when $N=1$. Note though that in contrast to our proof the methods in \cite{Darvas,DDL} do not rely on the calculation of the Monge-Amp\`ere mass of the geodesic ray, as this is automatically zero. Nevertheless our construction was initially inspired by these papers \cite{Darvas, DDL} and it would be very interesting to know if possibly there are more links between the results.

\subsection{Acknowledgements}

I want to thank Bo Berndtsson, Antonio Trusiani and especially Tam\'as Darvas for many helpful comments on an early draft. 

\section{Preliminaries} \label{Prel}

Let $(X,\omega)$ be a compact K\"ahler manifold and $\theta$ be a smooth real $(1,1)$-form on $X$. Recall from the introduction that a function $\phi:X\to [-\infty,\infty)$ is said to be $\theta$-psh if it is upper semicontinuous, $L^1_{loc}$ and $\theta+dd^c\phi\geq 0$ in the sense of currents, and that by the $dd^c$-lemma any closed positive $(1,1)$-current $T$ in the class $[\theta]$ can be written as $T=\theta+dd^c\phi$ where $\phi$ is $\theta$-psh.

The set of $\theta$-psh functions is denoted by $PSH(X,\theta)$. The cohomology class $[\theta]$ is called pseudoeffective if it contains a closed positive current, i.e. if $PSH(X,\theta)$ is nonempty, while $[\theta]$ is said to be big if for some $\epsilon>0$, $[\theta-\epsilon \omega]$ is pseudoeffective.

A $\theta$-psh function is said to have analytic singularities if locally it can be written as $c\ln(\sum_i |g_i|^2)+f$ where $c>0,$ $g_i$ is a finite collection of local holomorphic functions and $f$ is smooth. By a deep regularization result of Demailly \cite{Dem}, if $[\theta]$ is big then there exists a $\theta$-psh function with analytic singularities. Since any proper analytic subset is closed and complete pluripolar we note that if $\phi$ is less singular than some function with analytic singularities then $\phi$ has small unbounded locus. 

A $\theta$-psh function is said to have minimal singularities if it is less singular than all other $\theta$-psh functions.

We now come to the notion of non-pluripolar positive products of closed positive currents. This theory was first developed in the local setting by Bedford-Taylor \cite{BT82,BT87} and later in the geometric setting of compact K\"ahler manifolds by Boucksom-Eyssidieux-Guedj-Zeriahi \cite{BEGZ}.

If $T_i$, $i=\{1,...,p\}$ are closed positive $(1,1)$-currents there is a closed positive $(p,p)$-current $\langle T_1\wedge ...\wedge T_p\rangle$ called the non-pluripolar positive product of $T_i$ (see \cite{BEGZ} for the definition). The product is symmetric and multilinear \cite[Prop. 1.4(c)]{BEGZ})). Importantly non-pluripolar products never puts mass on pluripolar sets.

When $p=n=\dim_{\mathbb{C}} X$, $\langle T_1\wedge ...\wedge T_n\rangle$ is a positive measure, and when the $n$ currents are all equal $T_i=\theta+dd^c\psi$, then $\langle (\theta+dd^c\psi)^n\rangle$ is known as the (non-pluripolar) Monge-Amp\`ere measure of $\psi$, which we also denote by $MA_{\theta}(\psi)$. From the symmetry and multilinearity of the non-pluripolar product we get that if $\phi$ and $\psi$ are $\theta$-psh the measure $MA_{\theta}((1-t)\phi+t\psi)$ depends continuously on $t\in[0,1]$.

A basic fact is that if on some upen set $U$ $\psi$ is $C^{1,1}$ (or more generally if $dd^c\psi$ has coefficients in $L^{\infty}$) then 
\begin{equation} \label{explicit}
\mathbbm{1}_U MA_{\theta}(\psi)=\mathbbm{1}_U (\theta+dd^c u)^n.
\end{equation}
Here the right hand side simply denotes the measure one gets by taking the appropriate determinant of the coefficient functions. 

The following convergence result for Monge-Amp\`ere measures by Bedford-Taylor \cite[Thm. 2.1]{BT82} is absolutely fundamental.

\begin{theorem} \label{BTthm}
Let $U$ be an open set and $u_k$ be a decreasing sequence of $\theta$-psh functions such that $u:=\lim_{k\to \infty }u_k$ is locally bounded on $U$ ($u$ will then automatically by $\theta$-psh on $U$). Then on $U$ the measures $MA_{\theta}(u_k)$ converge weakly to $MA_{\theta}(u)$.
\end{theorem} 

Recall from \cite{BT87,BEGZ} that the plurifine topology is the one generalated by sets of the form $U\cap \{v>0\}$ where $U$ is open and $v$ is some psh function on $U$. Monge-Amp\`ere measures are local in the plurifine topology \cite[Prop. 1.4(a)]{BEGZ}, i.e. if $\phi=\psi$ on a plurifine open set $O$ then $\mathbbm{1}_OMA_{\theta}(\phi)=\mathbbm{1}_OMA_{\theta}(\psi)$ .

Let us state the full version of monotonicity result of Boucksom-Eyssidieux-Guedj-Zeriahi from \cite[Thm. 1.16]{BEGZ} mentioned in the Introduction.

\begin{theorem} \label{BEGZthm}
Assume we have two $p$-tuples of currents $T_i=\theta_i+dd^c\phi_i$ and $T_i'=\theta_i+dd^c\psi_i$ such that for each $i$ $\phi_i$ is less singular than $\psi_i$, and furthermore each $\phi_i$ and $\psi_i$ has small unbouded locus. Then the cohomology class $$\langle T_1\wedge ...\wedge T_p \rangle-\langle T_1'\wedge ...\wedge T_p' \rangle$$ is pseudoeffective, i.e. contains a closed positive $(p,p)$-current. In the particular case that $p=n$, $\theta_i=\theta$, $\phi_i=\phi$ and $\psi_i=\psi$ this means precisely that $$\int_X MA_{\theta}(\phi)\geq \int_XMA(\psi).$$
\end{theorem}

Boucksom-Eyssidieux-Guedj-Zeriahi also proved a comparison principle \cite[Cor. 2.3]{BEGZ}:

\begin{theorem}
For any two $\theta$-psh functions $\phi$ and $\psi$ we have $$\int_{\{\phi<\psi\}}MA_{\theta}(\psi)\leq \int_{\{\phi<\psi\}}MA_{\theta}(\phi)+\textrm{vol}([\theta])-\int_XMA_{\theta}(\phi).$$ If $\phi$ has minimal singularities then $\int_X MA_{\theta}(\phi)=\textrm{vol}([\theta])$ and thus $$\int_{\{\phi<\psi\}}MA_{\theta}(\psi)\leq \int_{\{\phi<\psi\}}MA_{\theta}(\phi).$$
\end{theorem}

In \cite[Rem. 2.4]{BEGZ} Boucksom-Eyssidieux-Guedj-Zeriahi also noted that when $\phi$ is less singular than $\psi$ and both have small unbounded locus, Theorem \ref{BEGZthm} can be used to show that $$\int_{\{\phi<\psi\}}MA_{\theta}(\psi)\leq \int_{\{\phi<\psi\}}MA_{\theta}(\phi).$$

Let us finally mention the domination principle \cite[Cor. 2.5]{BEGZ}:

\begin{theorem}
Let $\phi$ and $\psi$ be $\theta$-psh. If $\phi$ has minimal singularities and $\psi\leq \phi$ a.e. with respect to $MA_{\theta}(\phi)$ then $\psi\leq \phi$ everywhere.
\end{theorem} 

\section{A construction on $X\times \mathbb{P}^N$}

Assume $[\theta]$ is big.

Pick $N\in \mathbb{N}$ and let $\omega_{FS}$ denote the Fubini-Study form on $\mathbb{P}^N$. If $\pi_X$ and $\pi_{\mathbb{P}^N}$ are the projections from $X\times \mathbb{P}^N$ to $X$ and $\mathbb{P}^N$ respectively we let $\tilde{\theta}:=\pi_X^*\theta+\pi_{\mathbb{P}^N}^*\omega_{FS}$. We clearly have that $[\tilde{\theta}]$ also is big.  

Let $Z_i$ be the homogeneous coordinates on $\mathbb{P}^N$ and denote $$\ln|Z_i|^2_{FS}:=\ln\left(\frac{|Z_i|^2}{\sum_{j=0}^N|Z_j|^2}\right).$$ Then $\ln|Z_i|^2_{FS}$ is $\omega_{FS}$-psh with $\omega_{FS}+dd^c\ln|Z_i|^2_{FS}$ being the current of integration along the hyperplane $\{Z_i=0\}$. 

Let $\Sigma_N$ denote the $N$-dimensional unit simplex $\Sigma_N:=\{x\in \mathbb{R}^N: x_i\geq 0, \sum x_i\leq 1\}$. To ease notation we will write $|x|:=\sum_i x_i$. 

Since $[\theta]$ is big we can pick a $\phi_0\in PSH(X,\theta)$ with analytic singularities and we now define $$\Phi:=\sup^*_{x\in \Sigma_N}\{(1-|x|)(\phi_0+\ln|Z_0|^2_{FS})+|x|\phi+\sum_{i=1}^N x_i\ln|Z_i|^2_{FS}-\sum_{i=1}^Nx_i^2\}.$$ Here $*$ means that we take the usc regularization of the supremum. We see that $\Phi$ is a $\tilde{\theta}$-psh function and since $\phi_0+\ln|Z_0|^2_{FS}$ has analytic singularities and $\Phi\geq \phi_0+\ln|Z_0|^2_{FS}$ it follows that $\Phi$ has small unbounded locus.

Key to the proof of Theorem \ref{compthm} will be the following proposition.

\begin{proposition} \label{keyprop}
We have that $$\int_{X\times \mathbb{P}^N}MA_{\tilde{\theta}}(\Phi)=N\int_{t=0}^1\left(\int_XMA_{\theta}((1-t)\phi_0+t\phi)\right) t^{N-1}dt.$$
\end{proposition}  

To prove Proposition \ref{keyprop} we will use the following lemma.

\begin{lemma} \label{keylemma}
Let $U$ be an open set in $\mathbb{C}^n$, $u$ and $v$ two psh functions on $U$, and we assume that $u$ is smooth while $v$ is locally bounded. Let $\Phi$ be the psh function on $U\times \mathbb{C}^N$ defined as $$\Phi:=\sup^*_{x\in \Sigma_N}\{(1-|x|)u+|x|v+\sum_{i=1}^Nx_i\ln|z_i|^2-\sum_{i=1}^Nx_i^2\}.$$ Then we have that $$(\pi_U)_*MA(\Phi)=N\int_{t=0}^1MA((1-|x|)u+|x|v)t^{N-1}dt.$$ 
\end{lemma}

\begin{proof}
First we assume that $v$ is smooth.

Clearly $\Phi$ is invariant under the standard torus-action on $\mathbb{C}^N$ so if we write $y_i:=\ln|z_i|^2$ then for each $p\in U$, $f_p(y):=\Phi(p,z)$ is a convex function on $\mathbb{R}^N$. One notes that $f_p$ is $C^{1,1}$ (e.g. by observing that it is the Legendre transform of a strictly convex function and using \cite[Thm. 26.3]{Roc}) and since $f_p$ varies smoothly with $p$ it follows that $\Phi$ itself is $C^{1,1}$.

The gradient of $f_p$ will map $\mathbb{R}^N$ to $\Sigma_N$, indeed $\nabla f_p(y)=x$ iff $$f_p(y)=(1-|x|)u+|x|v+\sum_{i=1}^Nx_i\ln|z_i|^2-\sum_{i=1}^Nx_i^2.$$ Let us write $\hat x(p,z):=\nabla f_p(y)$, hence we have that $$\Phi(p,z)=(1-|\hat{x}|)u+|\hat{x}|v+\sum_{i=1}^N\hat{x}_i\ln|z_i|^2-\sum_{i=1}^N\hat{x}_i^2.$$ 

Let $V:=\hat{x}^{-1}(\Sigma_N^{\circ})$.  Using the fact that $\Phi$ is $C^{1,1}$ and $\partial V$ has measure zero we get that $MA(\Phi)=\mathbbm{1}_V(dd^c\Phi)^{n+N}+\mathbbm{1}_{(V^c)^{\circ}}(dd^c\Phi)^{n+N}$. Now $f_p$ fails to be strictly convex outside of $\nabla f^{-1}(\Sigma_N^{\circ})$ thus $\mathbbm{1}_{(V^c)^{\circ}}(dd^c\Phi)^{n+N}=0$, i.e. $MA(\Phi)=\mathbbm{1}_V(dd^c\Phi)^{n+N}$. 

On $V$ we have $$\hat{x}_i(p,z)=\frac{1}{2}(v(p)-u(p)+\ln|z_i|^2)$$ and so there $$\Phi(p,z)=u+\sum_i \hat{x}_i((v-u+\ln|z_i|^2-\hat{x}_i))=u+\sum_i \hat{x}_i^2.$$ We get then on $V$ 
\begin{eqnarray*}
dd^c\Phi=dd^cu+\sum_i \hat{x}_i(dd^c(v-u)-dd^c\hat{x}_i)+\sum_i d\hat{x}_i\wedge d^c \hat{x}_i+\sum_i \hat{x}_idd^c\hat{x}_i=\\=(1-|\hat{x}|)dd^cu +|\hat{x}|dd^cv +\sum_i d\hat{x}_i\wedge d^c \hat{x}_i.
\end{eqnarray*}

Since $(\sum_i d\hat{x}_i\wedge d^c \hat{x}_i)^k=0$ for any $k>N$ it follows that 
\begin{equation} \label{mongeampere2}
MA(\Phi)=\mathbbm{1}_V((1-|\hat{x}|)dd^cu+|\hat{x}|dd^cv)^n\wedge (\sum_i d\hat{x}_i\wedge d^c \hat{x}_i)^N.
\end{equation}

Let $V_p:=V\cap (\{p\}\times \mathbb{C}^N)$. Since $dd^cu$ and $dd^cv$ do not contain any terms with $dz_i$ or $d\bar{z}_i$, only the derivatives of $\hat{x}_i$ in the $V_p$-directions will enter into the expression of $MA(\Phi)$. One easily checks that $$(\sum_i d(\hat{x}_i)_{|V_p}\wedge d^c (\hat{x}_i)_{|V_p})^N=MA(\Phi_{|V_p})$$ and hence from (\ref{mongeampere2}) we get 
\begin{equation} \label{ma3}
MA(\Phi)=\mathbbm{1}_V((1-|\hat{x}|)dd^cu+|\hat{x}|dd^cv)^n\wedge MA(\Phi_{|V_p}).
\end{equation}

It is standard fact that if we let $Y$ denote the map $z\mapsto y$ then 
\begin{equation} \label{moment1}
Y_*(MA(\Phi_{|\{p\}\times\mathbb{C}^N}))=N!MA_{\mathbb{R}^N}(f_p).
\end{equation}

It is also well known that the real Monge-Amp\`ere measure is connected with the gradient map so that
\begin{equation} \label{moment2}
(\nabla f_p)_*(MA_{\mathbb{R}^N}(f_p))=dx_{|\Sigma_N}.
\end{equation}

let $\mu$ be the "moment" map from $U\times \mathbb{C}^N$ to $U\times \Sigma_N$ given by $\mu(w,z):=(w,\hat{x}(w,z))$. Using (\ref{ma3}), (\ref{moment1}) and (\ref{moment2}) then gives us that 
\begin{equation}  \label{eqeq1}
\mu_*(MA(\Phi))=MA_U((1-|x|)u+|x|v)N!dx_{|\Sigma^{\circ}_N}.
\end{equation}

If we let $\pi'$ denote the projection from $U\times \Sigma_N$ to $U$ we can write $\pi_U=\pi'\circ \mu$ and hence by (\ref{eqeq1}) 
\begin{eqnarray} \label{comeon}
(\pi_U)_*MA(\Phi)=\pi'_*MA((1-|x|)u+|x|v)N!dx_{|\Sigma_N}= \nonumber \\=N!\int_{\Sigma_N}MA((1-|x|)u+|x|v)dx.
\end{eqnarray}
By homogeneity the volume of $\Sigma_N\cap\{|x|\leq t\}$ is $t^N/N!$ and it follows that $$N!\int_{\Sigma_N}MA((1-|x|)u+|x|v)dx=N\int_{t=0}^1MA((1-t)u+tv)t^{N-1}dt,$$ with (\ref{comeon}) proving the Lemma in the case when $v$ is smooth.

Recall that on $V$ we had that $$\hat{x}_i(p,z)=\frac{1}{2}(v(p)-u(p)+\ln|z_i|^2).$$ Since clearly $0\leq \hat{x}_i\leq 1$ this implies that if $v>u-C$ for some constant $C$ then the closure of $V$ is contained in $U\times e^{C/2+1}\mathbb{D}^N$. But recall that $MA(\Phi)$ was supported on the closure of $V$, hence $MA(\Phi)$ is supported on $U\times e^{C/2+1}\mathbb{D}^N$ where $e^{C/2+1}\mathbb{D}^N$ denotes the polydisc $\{z: \forall i, |z_i|<e^{C/2+1}\}$.

Now let $v$ be just locally bounded. Without loss of generality we can assume that $v$ is in fact bounded, and that say $v>u-C$ for some contant $C$.  

Let $v_j$ be a sequence of smooth psh functions on $U$ decreasing to $v$ and write $$\Phi_j:=\sup^*_{x\in \Sigma_N}\{(1-|x|)u+|x|v_j+\sum_{i=1}^Nx_j\ln|z_j|^2-\sum_{i=1}^Nx_i^2\}.$$ By what we have established 
\begin{equation} \label{pushprop}
(\pi_U)_*MA(\Phi_j)=N\int_{t=0}^1MA((1-t)u+tv_j)t^{N-1}dt.
\end{equation}
We have that $\Phi_j$ decreases to $\Phi$ and so by Theorem \ref{BTtheorem} $MA(\Phi_j)$ will converge weakly to $MA(\Phi)$. Also note that $v_j\geq v>u-C$ so as a consequence each $MA(\Phi_j)$ is supported on $U\times e^{C/2+1}\mathbb{D}^N$, which implies that $(\pi_U)_*MA(\Phi_j)$ corverge weakly to $(\pi_U)_*MA(\Phi)$. Since again by Theorem \ref{BTthm} each $MA((1-t)u+tv_j)$ converge weakly to $MA((1-t)u+tv)$, this proves the Lemma.  

\end{proof}

We now prove Proposition \ref{keyprop}.

\begin{proof}
We wish to show that
\begin{equation} \label{pushpush}
(\pi_X)_*MA_{\tilde{\theta}}(\Phi)=N\int_{t=0}^1MA_{\theta}((1-t)\phi_0+t\phi)t^{N-1}dt.
\end{equation}

The Monge-Amp\`ere measure $MA_{\tilde{\theta}}(\Phi)$ will not charge the analytic set $\{\phi_0=-\infty\}$ so let us pick a coordinate chart $U\subseteq X\setminus \{\phi_0=-\infty\}$ where $\theta=dd^ch$ for some smooth function $h$. Then $u:=\phi_0+h$ is a smooth psh function on $U$ while $v:=\phi+h$ is simply psh on $U$. We thus need to show that 
\begin{equation} \label{eqpush}
(\pi_U)_*MA(\Phi+h)=N\int_{t=0}^1MA((1-t)u+tv)t^{N-1}dt
\end{equation} 
where $$\Phi+h=\sup^*_{x\in \Sigma_N}\{(1-|x|)u+|x|v+\sum_{i=1}^Nx_i\ln|z_i|^2-\sum_{i=1}^Nx_i^2\}.$$ 

Pick a constant $C$ and let $\phi_C:=\max(\phi,u-C)$, $v_C:=\phi_C+h$ and $$\Phi_C+h=\sup^*_{x\in \Sigma_N}\{(1-|x|)u+|x|v_C+\sum_{i=1}^Nx_i\ln|z_i|^2-\sum_{i=1}^Nx_i^2\}.$$ Since $u$ is smooth and $v_C$ locally bounded on $U$ Lemma \ref{keylemma} says that $$(\pi_U)_*MA(\Phi_C+h)=N\int_{t=0}^1MA((1-t)u+tv_C)t^{N-1}dt.$$ By the fact that the Monge-Amp\`ere measures are local in the plurifine topology we thus get that $$\mathbbm{1}_{\{\phi>u-C\}}(\pi_U)_*MA(\Phi+h)=\mathbbm{1}_{\{\phi>u-C\}}N\int_{t=0}^1MA((1-t)u+tv)t^{N-1}dt.$$ Observing that neither $(\pi_U)_*(MA(\Phi+h))$ nor $N\int_{t=0}^1MA((1-t)u+tv)t^{N-1}dt$ puts any mass on the pluripolar set $\{v=-\infty\}$ we get (\ref{eqpush}) by letting $C\to \infty$. This establishes (\ref{pushpush}) while integrating it over $X$ finally yields the Proposition.

\end{proof}

\section{Proofs of main results}

We can now prove Theorem \ref{compthm}. 

\begin{proof}
Let $\phi$ and $\psi$ be as in the statement of Theorem \ref{compthm}. By \cite[Prop. 1.22]{BEGZ}, if $[\theta]$ is not big then both Monge-Amp\`ere masses are zero, thus we can assume that $[\theta]$ is big.

Pick a large $N$ and let $\Phi$ and $\Psi$ be defined as above. From the construction it is clear that $\Phi$ is less singular than $\Psi$. Since they also have small unbounded locus we know from \cite[Thm. 1.16]{BEGZ} that $$\int_{X\times \mathbb{P}^N}MA_{\tilde{\theta}}(\Phi)\geq \int_{X\times \mathbb{P}^N}MA_{\tilde{\theta}}(\Psi).$$ Combined with Proposition \ref{keyprop} we get that 
\begin{eqnarray} \label{ineq1}
N\int_{t=0}^1\left(\int_XMA_{\theta}((1-t)\phi_0+t\phi)\right) t^{N-1}dt\geq  \nonumber \\ \geq N\int_{t=0}^1\left(\int_XMA_{\theta}((1-t)\phi_0+t\psi)\right) t^{N-1}dt.
\end{eqnarray}
Recall from Section \ref{Prel} that the function $$g(t):=\int_XMA_{\theta}((1-t)u+t\phi)$$ is continuous in $t\in [0,1]$. It follows that 
\begin{equation} \label{lim1}
\lim_{N\to \infty}N\int_{t=0}^1\left(\int_XMA_{\theta}((1-t)\phi_0+t\phi)\right) t^{N-1}dt=\int_XMA_{\theta}(\phi),
\end{equation}
and similarly for $\psi$. The theorem follows from combining (\ref{ineq1}) and (\ref{lim1}).
\end{proof}

Let us recall and prove the corollary stated in the Introduction.

\begin{corollary} \label{corcomp2} 
Let $\phi$ and $\psi$ be $\theta$-psh and assume $\phi$ is less singular than $\psi$. Then $$\int_{\{\phi<\psi\}}MA_{\theta}(\psi)\leq \int_{\{\phi<\psi\}}MA_{\theta}(\phi).$$
\end{corollary}

\begin{proof}
Here we precisely follow \cite[Cor. 2.3, Rmk. 2.4]{BEGZ}. 

Let $\epsilon>0$ and $\phi_{\epsilon}:=\max(\phi,\psi-\epsilon)$. By Theorem \ref{compthm} we have that $\int_XMA_{\theta}(\phi_{\epsilon})=\int_XMA_{\theta}(\phi)$ and so using the plurifine locality of Monge-Amp\`ere measures we get 
\begin{eqnarray*}
\int_XMA_{\theta}(\phi)=\int_XMA_{\theta}(\phi_{\epsilon})\geq \int_{\{\phi<\psi-\epsilon\}}MA_{\theta}(\phi_{\epsilon})+\int_{\{\phi>\psi-\epsilon\}}MA_{\theta}(\phi_{\epsilon})=\\
=\int_{\{\phi<\psi-\epsilon\}}MA_{\theta}(\psi)+\int_{\{\phi>\psi-\epsilon\}}MA_{\theta}(\phi)=\\
=\int_{\{\phi<\psi-\epsilon\}}MA_{\theta}(\psi)+\int_XMA_{\theta}(\phi)-\int_{\{\phi\leq\psi-\epsilon\}}MA_{\theta}(\phi),
\end{eqnarray*}
and hence $$\int_{\{\phi<\psi-\epsilon\}}MA_{\theta}(\psi)\leq \int_{\{\phi\leq\psi-\epsilon\}}MA_{\theta}(\phi)\leq \int_{\{\phi<\psi\}}MA_{\theta}(\phi).$$ The result now follows from letting $\epsilon$ tend to zero.
\end{proof}

We will also mention a second corollary, namely the following domination principle:

\begin{corollary} \label{cordom2}
Let $\phi$, $\psi$ and $\rho$ be $\theta$-psh, and assume that $\phi$ is less singular than $\psi$ and $\rho$. Then if $\phi\geq \psi$ a.e. with respect to $MA(\phi)$ it follows that $\phi\geq \psi$ a.e. also with respect to $MA(\rho)$.
\end{corollary}

\begin{proof}
Here we precisely follow \cite[Cor. 2.5]{BEGZ}.

We can assume that $\rho\leq \phi$. Let $\epsilon>0$. By Corollary \ref{corcomp2} we get that 
\begin{eqnarray} \label{finalfinal}
\epsilon^n\int_{\{\phi<(1-\epsilon)\psi+\epsilon\rho\}}MA_{\theta}(\rho)\leq \int_{\{\phi<(1-\epsilon)\psi+\epsilon\rho\}}MA_{\theta}((1-\epsilon)\psi+\epsilon\rho)\leq \nonumber \\ \leq \int_{\{\phi<(1-\epsilon)\psi+\epsilon\rho\}}MA_{\theta}(\phi).
\end{eqnarray}
But since $\{\phi<(1-\epsilon)\psi+\epsilon\rho\}\subseteq \{\phi<\psi\}$ it follows that $$\int_{\{\phi<(1-\epsilon)\psi+\epsilon\rho\}}MA_{\theta}(\phi)=0$$ and hence by (\ref{finalfinal}) $$\int_{\{\phi<(1-\epsilon)\psi+\epsilon\rho\}}MA_{\theta}(\rho)=0.$$ Letting $\epsilon$ tend to zero yields the result.
\end{proof}

\end{document}